\documentclass[10pt,a4paper,oneside]{amsart}

\usepackage{latexsym, amsthm}

\usepackage{amsfonts}

\usepackage{hyperref}

\usepackage{euscript,mathrsfs}

\newtheorem{theorem}{Theorem}[section]
\newtheorem{lemma}[theorem]{Lemma}

\newtheorem{corollary}[theorem]{Corollary}

\theoremstyle{definition}

\newtheorem{example}[theorem]{Example}

\theoremstyle{remark}
\newtheorem{remark}[theorem]{Remark}
\def\Fq{{\mathbb F}_q}
\def\a{{\alpha}}

\def\AA{{\mathbb A}}

\def\PP{{\mathbb P}}
\def\dP{\widehat{\mathbb P}}
\def\Z{{\mathbb Z}}

\newcommand{\RM}{\mathrm{RM}}
\newcommand{\PRM}{\mathrm{PRM}}

\def\PP{{\mathbb P}}

\newcommand{\w}{{\mathrm{w_H}}}
\begin{document}

\title[Hypersurfaces over finite fields]{On a conjecture of Tsfasman and an inequality of Serre for the number of points of hypersurfaces over finite fields}


\dedicatory{Dedicated to Misha Tsfasman and Serge Vl\u{a}du\c{t} on their $60^{\rm th}$ birthdays}

\author{Mrinmoy Datta}
\address{Department of Mathematics,
Indian Institute of Technology Bombay,\newline \indent
Powai, Mumbai 400076, India.}
\curraddr{}
\email{mrinmoy.dat@gmail.com}
\thanks{The first named author is partially supported by a doctoral fellowship from the National Board for Higher Mathematics, a division of the Department of Atomic Energy, Govt. of India.}

\author{Sudhir R. Ghorpade}
\address{Department of Mathematics, 
Indian Institute of Technology Bombay,\newline \indent
Powai, Mumbai 400076, India.}
\email{srg@math.iitb.ac.in}
\thanks{The second named author is partially supported by Indo-Russian project INT/RFBR/P-114 from the Department of Science \& Technology, Govt. of India and  IRCC Award grant 12IRAWD009 from IIT Bombay.}


\date{\today}

\begin{abstract}
We give a short proof of an inequality, conjectured by Tsfasman and proved by Serre, for the maximum number of points of hypersurfaces over finite fields. Further, we consider a conjectural extension, due to Tsfasman and Boguslavsky,  of this inequality to an explicit formula for the maximum number of common solutions of a system of linearly independent multivariate homogeneous polynomials of the same degree with coefficients in a finite field.  This conjecture is shown to be false, in general, but is also shown to hold in the affirmative in a special case. Applications to generalized hamming weights of projective Reed-Muller codes are outlined and a comparison with an older conjecture of Lachaud and a recent result of Couvreur is given. 
\end{abstract}

\maketitle


\section{Introduction}
What is the maximum number of $\Fq$-rational points that a hypersurface of degree $d$ in $m$-space over the finite field $\Fq$ with $q$ elements can have? An intuitive approach could be to project the $m$-space onto an $(m-1)$-space. If this projection map from
the hypersurface to the $(m-1)$-space is flat, then a point below has at most $d$ points above on the hypersurface. This suggests that $d$ times the number of $\Fq$-rational points in the $(m-1)$-space is a natural upper bound. More precisely, if $f\in \Fq[x_1, \dots , x_m]$ is 
of degree $d$ and $Z(f):=\{P\in \AA^m(\Fq): f(P)=0\}$ the corresponding affine hypersurface or if $F\in \Fq[x_0, x_1, \dots , x_m]$ is a (nonzero) homogeneous polynomial of degree $d$ and  $V(F):=\{P\in \PP^m(\Fq): F(P)=0\}$ the corresponding projective hypersurface, then 
\begin{equation}
\label{BasicBound}
|Z(f)| \le dq^{m-1} \quad \text{ and } \quad |V(F)| \le dp_{m-1},
\end{equation}
where for any $j\in \Z$, we have set 
\begin{equation}
\label{pm}
p_j := |\PP^j(\Fq)| = q^j + q^{j-1} + \dots + q + 1 \text{ if $j\ge 0$ and  $p_j := 0$ if $j<0$.}
\end{equation} 
The bounds in \eqref{BasicBound} are true and a precise proof can be easily given using double induction on $d$ and $m$; see, e.g., \cite[pp. 275--276]{LN}. Variants of the bound in  \eqref{BasicBound} for $|Z(f)|$ are also known in the literature as Schwarz-Zippel Lemma or more elaborately, as Schwarz-Zippel-DeMillo-Lipton Lemma, although it can be traced more than 50 years earlier to the work of O. Ore in 1922; see, e.g., \cite[p. 320]{LN}. 
 If $d>q$,  the bound $dq^{m-1}$ exceeds $|\AA^m(\Fq)|$ and is hence uninteresting, whereas if $d\le q$, then it is attained as can be seen readily by considering the polynomial $g_d(x_1, \dots , x_m) = (x_1-a_1) \cdots (x_1 - a_d)$, where $a_1, \dots , a_d$ are distinct elements of $\Fq$.  Thus $dq^{m-1}$ is the maximum value for $|Z(f)|$ when $d\le q$.  
Likewise in the projective case,  if $d\ge q+1$, 
the bound  $d p_{m-1}$ exceeds   $|\PP^m(\Fq)|$ and so it  is uninteresting in that case. 
However, if $d\le q$ and $a_1, \dots , a_d$ are as before, then it is easy to see that the analogous homogeneous 
polynomial $G_d(x_0,x_1, \dots , x_m) = (x_1-a_1x_0) \cdots (x_1 - a_dx_0)$ has exactly $dq^{m-1} + p_{m-2}$ zeros in $\PP^m(\Fq)$. Also if $d=q+1$, then $G_{q+1}:= x_0G_q$ has exactly $dq^{m-1} + p_{m-2}$
zeros in $\PP^m(\Fq)$. But  $dq^{m-1} + p_{m-2} < dp_{m-1}$  if $d> 1$ and $m>1$. Moreover, an example of a homogeneous polynomial of degree $d\le q$ with exactly $dp_{m-1}$ zeros is difficult to come by. Motivated perhaps by this, M. A. Tsfasman made a conjecture, in the late 80's, that $dq^{m-1} + p_{m-2}$ is the ``true upper bound''. In other words,
\begin{equation}
\label{SerreIneq}
 |V(F)| \le dq^{m-1} + p_{m-2}  \ \text{ and hence } \ \max_F  |V(F)| = dq^{m-1} + p_{m-2}  \ \text{ if } d\le q+1,
\end{equation}
where $F$ varies over homogeneous polynomials of degree $d$ in $\Fq[x_0, x_1, \dots , x_m]$. 
This conjecture was soon proved in the affirmative by J.-P. Serre \cite{Se}, and thus we will refer to the inequality in \eqref{SerreIneq} as Serre's inequality. 
An alternative proof of \eqref{SerreIneq} was also given later by S{\o}rensen \cite{So}. 
Some years later, Tsfasman and Boguslavsky  \cite{Bog} formulated more general conjectures for the maximum number of points in $\PP^m(\Fq)$ of systems of polynomials equations of the form
\begin{equation}
\label{Fr}
\left\{ \begin{array}{l} F_1(x_0, x_1, \dots , x_m) = 0 \\
\vdots \\
F_r(x_0, x_1, \dots , x_m) = 0 \end{array} \right.
\end{equation}
A quantitative version of their conjecture essentially states that if $d < q-1$, then 
\begin{equation}
\label{TBr}
\max_{F_1, \dots , F_r} |V(F_1, \dots , F_r) | = \displaystyle{   p_{m-2j} + \sum_{i=j}^m} \nu_i (p_{m-i} - p_{m-i-j}),
\end{equation}
where $F_1, \dots , F_r$ vary over linearly independent homogeneous polynomials of degree $d$ in 
$\Fq[x_0, x_1, \dots , x_m]$, and 
where $(\nu_1, \dots, \nu_{m+1})$ is the $r$th element of the set
$$
\Lambda(d,m):=\left\{(\alpha_1, \dots , \alpha_{m+1}) \in \Z^{m+1} : \alpha_1, \dots , \alpha_{m+1} \ge 0 \text{ and }  \alpha_1+ \cdots + \alpha_{m+1} = d\right\}
$$
ordered in descending lexicographic order, 
and where $j := \min\{i : \nu_i \ne 0\}$. 

Observe that $(d,0,\dots , 0)$ is lexicographically the first element of $\Lambda(d,m)$
and so 
\eqref{TBr} reduces to 
\eqref{SerreIneq} in the case $r=1$, and is thus true, thanks to Serre \cite{Se} and S{\o}rensen \cite{So}. The conjecture was proved in the affirmative in the next case of $r=2$ by Boguslavsky in 1997. Also 
\eqref{TBr} holds trivially when $d=1$. But the general case appears to have been open for almost two decades. Moreover, some auxiliary conjectures given in \cite{Bog} (see also Remark \ref{ThreeConj} in Section~\ref{sec3} below) that would imply the Tsfasman-Boguslavsky Conjecture have also been open. 

We are now ready to describe the main results of this paper. In Section \ref{sec2} below, we give a very short proof of Serre's inequality. 
Serre's original proof is quite elegant and S{\o}rensen's proof has its merits and applications. 
Both the proofs  involve some clever double counting argument, which are avoided in the short proof 
given here,  
thus making it comparable to the elementary proofs of \eqref{BasicBound} that one can find  
in textbooks \cite[Theorems 6.13, 6.15]{LN}. Next, we consider the Tsfasman-Boguslavsky Conjecture (TBC) stated above and show that while it 
is true if 
$d=2$ and $r\le m+1$, it is false, in general, when $d=2$ and $r\ge m+2$. In a forthcoming paper \cite{DG2} it is shown that the TBC 
is true for any positive integer $d$, provided $r \le m+1$. In the last section, we outline connections with coding theory and show that the TBC and in particular, our results are intimately related to the explicit determination of generalized hamming weights of projective Reed-Muller codes. We also compare the TBC with a recent result of Couvreur \cite{C} on an upper bound for the number of $\Fq$-rational points of arbitrary projective varieties defined over $\Fq$. 

A key ingredient for us is the work of Zanella \cite{Z} on the number of $\Fq$-rational points on intersections of quadrics or equivalently, on linear sections of the 
variety defined by the quadratic Veronese embedding  $\PP^m \hookrightarrow \PP^M$, where $M= {\binom{m+2}{2}} - 1$.

\section{Serre's Inequality}
\label{sec2}

Throughout this paper, $m$ denotes a positive integer, $q$ a prime power, $\Fq$ the field with $q$ elements. Also for any nonnegative integer $j$, we will denote by $\PP^j$ the $j$-dimensional projective space over $\Fq$, and by $\dP^j$ its dual, consisting of all hyperplanes in $\PP^j$. Note that $|\PP^j (\Fq)| = |\dP^j (\Fq)| = p_j$, where $p_j$ is as in \eqref{pm}. 
The following lemma is due to Zanella \cite[Lemma 3.3]{Z}. A proof is included for the sake of completeness. Alternative proofs are indicated in Remark \ref{rem:ZanellaLemma} below. 

\begin{lemma}
\label{Zanella}
Let 
$X \subseteq \mathbb{P}^{m}$ and 
$a:=\max\{ |X \cap H| : H \in \dP^m\}$. Then 
$|X| \le a q + 1$.
\end{lemma}
\begin{proof}
Induct on $m$.  
The case $m=1$ being trivial, assume that $m>1$ and that the result holds for smaller values of $m$.
Let  $H^* \in \dP^m$ be such that $|H^* \cap X| = a.$ Let $Y:= H^* \cap X$ and let 
$H'$ be a hyperplane of $H^*$ such that $b := |H' \cap Y|$ is the maximum among the cardinalities of all hyperplane sections of $Y$ in $H^*\simeq \PP^{m-1}$. By induction hypothesis, $a \le bq+1$. Now we write 
$|X| = |H' \cap X| + |X \setminus (H' \cap X)|$ and observe that on the one hand $|H' \cap X| = |H' \cap Y| \le b$, whereas on the other hand, every $P\in  X \setminus (H' \cap X)$ is contained in a unique 
$H \in \dP^m$ with $H \supset H'$. Thus 
$$
\left |X \setminus (H' \cap X) \right| \le \left| \bigcup_{H \supset H'}  (H \cap X) \setminus (H'\cap X) \right| \le
(q+1) (a-b), 
$$
where the last inequality follows by noting that the number of $H \in \dP^m$ containing a fixed codimension $2$ linear subspace $H'$ of $\PP^{m}$ is $q+1$ and also noting that $\left| H \cap X\right|  - \left| H'\cap X \right| \le a- |H' \cap Y| = a-b$ for any such $H$.   Since $a \le bq+1$, it follows that 
$|X| \le  b+ (q + 1)(a - b)  = aq + a - bq \le aq+ 1$, as desired. 
\end{proof}

\begin{theorem}[Serre]  
\label{Serre} 
If $F\in \Fq[x_0, x_1, \dots , x_m]$ is homogeneous of degree $d$, then 
$ |V(F)| \le dq^{m-1} + p_{m-2}$. Consequently,  $\displaystyle{ \max_F  |V(F)| = dq^{m-1} + p_{m-2}}$  if $d\le q+1$. 
\end{theorem}

\begin{proof}
We use induction on $m$. The case $m=1$ is trivial. Thus we assume $m>1$ and 
consider the following two cases. 

\noindent
{\bf Case 1:} $V(F)$ does not contain any hyperplane in $\PP^m$. 

In this case $F|_H \neq 0$ for 
all $H \in \dP^{m}$. So the validity of the result 
with $m$ replaced by $m-1$ implies 
$|V(F) \cap H| \le dq^{m-2} + p_{m-3}$ for 
all $H \in \dP^{m}$. 
Hence by Lemma~\ref{Zanella}, $|V(F)| \le dq^{m-1} + p_{m-2}$.
Thus the desired inequality follows by induction on $m$.

\noindent
{\bf Case 2:} $V(F)$ contains a hyperplane in $\PP^m$, say $H = V(h)$. 

In this case $|V(F) \cap H| = |H| = p_{m-1}$. We will now estimate $|V(F) \cap H^c|$. First,  by a suitable linear change of  coordinates in $\PP^m$ assume that $h = x_0$. Since $F_{|H}=0$, we can write $F = x_0G$ 
 for some homogeneous $G\in \Fq[x_0, \dots ,x_m]$ of degree ${d-1}$. Moreover $V(F) \cap H^c$
corresponds to the zeros in $\AA^m(\Fq)$ of the polynomial $G(1, x_1, \dots , x_m)$. Hence 
$|V(F) \cap H^c| \le (d-1)q^{m-1}$, thanks to \eqref{BasicBound}. Consequently, 
$$
 |V(F)| =   |V(F) \cap H^c| + |V(F) \cap H|  \le (d-1)q^{m-1} + p_{m-1} =dq^{m-1} + p_{m-2} . 
$$
The assertion about $\displaystyle{ \max_F  |V(F)| }$ follows from the example of $G_d$ given earlier. 
\end{proof}

\begin{remark} 
\label{rem:ZanellaLemma}
Variants of Lemma \ref{Zanella} are seen elsewhere in the literature. For example, with notation as in  Lemma \ref{Zanella}, Homma \cite[Prop. 2.2]{H} has proved the following:
$$
|X| \le (a-1)q + 1 + \left\lfloor \frac{a-1}{p_{m-2}} \right\rfloor.
$$ 
To see that this implies $|X|\le aq+1$, 
observe that $a\le p_{m-1} = qp_{m-2}+1$ and thus
$ \left\lfloor {(a-1)}/{p_{m-2}} \right\rfloor \le q$. We can also deduce Lemma \ref{Zanella} from the Plotkin bound \cite[Thm. 1.1.45]{TVN} of coding theory. Indeed, 
we may assume without loss of generality that $X$ is not contained in a hyperplane of $\PP^m$ (for otherwise, $|X| = a < aq+1$). Now by \cite[Thm. 1.1.6]{TVN},  $X \subset \PP^m$ corresponds to a nondegenerate linear $[n,k,d]_q$-code $C$ 
with $n= |X|$, $k= m+1$ and $d = n - a$. 
The Plotkin bound gives 
$$
d \le \frac{nq^k(q-1)}{(q^k-1)q} = \frac{nq^{m}}{p_{m}},  \quad \text{which implies} \quad n  \frac{p_{m-1}}{p_{m}} \le a, \text{ that is, } n \le  \frac{ap_{m}}{p_{m-1}} \le aq+1. 
$$
In our proofs of  Lemma \ref{Zanella} and Theorem \ref{Serre}, we have avoided any double counting argument. But if we were to permit it, 
then Lemma \ref{Zanella} can be proved more easily by 
counting the set 
$
\{ (P, H) \in X \times \dP^m : P \in H \} 
$
in two ways using the first and the second projections, which yields $|X| p_{m-1} \le a p_m$, 
and hence $|X|\le aq+1$.  
\end{remark}

\section{Tsfasman-Boguslavsky Conjecture for Quadrics} 
\label{sec3}

In this section, we will consider the conjectural formula \eqref{TBr} for systems \eqref{Fr} where each $F_i$ is homogeneous of degree $2$. In other words, we consider intersections of quadrics. 
To begin with, observe that the maximum 
number of linearly independent homogeneous polynomials in $\Fq[x_0, \dots , x_m]$ of degree $2$ is $\delta_{m}$, where 
\begin{equation}
\label{delta}
\delta_{j} : = \binom{j+2}{2} = 1+ 2 + \cdots + (j+1) \quad \text{ for any $j\in \Z$ with } j \ge -1.
\end{equation}
Note that $0= \delta_{-1} < 1 = \delta_0 < \delta_1 < \dots < \delta_m$. 
 The following result is a restatement of \cite[Thm. 3.4]{Z}. 
As usual, $ \lfloor c \rfloor$ denotes  the integer part of a real number $c$.
 
\begin{theorem}[Zanella]  
\label{ZanellaThm} 
Let $r$ be a positive integer $\le \delta_m$, and let $F_1, \dots , F_r$ be linearly independent homogeneous polynomials in $\Fq[x_0, \dots , x_m]$ of degree $2$. If $k$ is the unique integer such that 
$-1\le k < m$ and $\delta_m - \delta_{k+1} < r \le \delta_{m} - \delta_{k}$, then 
\begin{equation}
\label{Zr}
|V(F_1, \dots , F_r) | \le p_k + \lfloor q^{\epsilon -1} \rfloor, \quad \text{ where } \quad \epsilon: =  \delta_{m} - \delta_{k} - r.
\end{equation}
\end{theorem}

We shall 
deduce from Theorem \ref{ZanellaThm} that if $d=2$, then the TBC
is true when $r \le m+1$ and false, in general, when $r > m+1$ and $m>2$. 

\begin{corollary}
\label{cor1}
Suppose $d=2$ and 
$1\le r \le m+1$. Then \eqref{TBr} holds. In particular, 
$$
\max_{F_1, \dots , F_r} |V(F_1, \dots , F_r) | = p_{m-1} + \lfloor q^{m -r} \rfloor = \begin{cases} p_{m-1} + q^{m-r} & \text{ if } r=1, \dots , m, \\ p_{m-1} & \text{ if } r= m+1. \end{cases}
$$
where the maximum is over linearly independent homogeneous polynomials $F_1, \dots , F_r$ of degree $d$ in 
$\Fq[x_0, x_1, \dots , x_m]$.
\end{corollary}

\begin{proof}
For $1 \le i \le m+1$, denote by $e_i$ the $(m+1)$-tuple with $1$ in $i$th place and $0$ elsewhere. 
Lexicographically, the first $m+1$ exponent vectors of monomials of degree $2$ in $m+1$ variables are 
$e_1+e_i$ for $i=1, 2, \dots , m+1$. Thus if $1\le r \le m$, then the corresponding Tsfasman-Boguslavsky bound in \eqref{TBr} is given by 
$$
T_r = p_{m-2} +  \left( p_{m-1} - p_{m-2} \right) +  \left( p_{m-r} - p_{m-r-1} \right)  =    p_{m-1} + q^{m-r} ,
$$
whereas if $r=m+1$, then $T_r = p_{m-1} =  p_{m-1} + \lfloor q^{m -r} \rfloor$. On the other hand, if $r\le m+1$, then the unique $k\in \Z$ satisfying $-1\le k < m$ and $\delta_m - \delta_{k+1} < r \le \delta_{m} - \delta_{k}$ is clearly $k = m-1$, and hence the corresponding Zanella bound in \eqref{Zr} is 
$$
Z_r =  p_{m-1} + \lfloor q^{\delta_{m} - \delta_{m-1} - r -1} \rfloor = p_{m-1} + \lfloor q^{m -r} \rfloor = T_r.
$$
Thus 
Theorem \ref{ZanellaThm} implies that $|V(F_1, \dots , F_r) | \le T_r$ for $1\le r \le m+1$. It remains to show that the bound $T_r$ is attained if $r \le m+1$. 
To this end, consider
$$
 Q_i = x_0 x_i \; \text{ for } \; i =1, \dots , m \quad \text{ and } \quad Q_{m+1}:= x_0^2. 
$$
For any $r\le m+1$, it is clear that  $Q_1, \dots , Q_r$ are linearly independent homogeneous polynomials of degree~$2$ in $\Fq[x_0, \dots , x_m]$. 
If $r\le m$, then the common zeros in $\PP^m$ of $Q_1, \dots , Q_r$ have homogeneous coordinates   
of the type $(a_0:a_1: \dots : a_m)$, where either (i) $a_0=0$ or (ii) $a_0 \ne 0$ and  $a_1 = \dots = a_r=0$. Hence 
$$
|V(Q_1, \dots , Q_r) | =  p_{m-1} + q^{m-r} = T_r \quad \text{ if } 1\le r \le m. 
$$
In case $r=m+1$, the common zeros of $Q_1, \dots , Q_r$ are only of the first type  
and so $|V(Q_1, \dots , Q_{m+1}) | =  p_{m-1} = T_{m+1}$. This completes the proof. 
\end{proof}

\begin{remark} 
\label{lastwt}
Roughly speaking, Corollary \ref{cor1} shows that \eqref{TBr} holds for small values of $r$. It is not difficult to show that it also holds for a few large values of $r$. For instance, when $r= \delta_m$, the intersection of $r$ linearly independent quadrics in $\PP^m$ 
is empty and thus has $0$ elements, quite in accordance with \eqref{TBr}. A little 
more generally, if $m\ge 2$, then one can see that the last $5$ exponent vectors of monomials of degree $2$ in $m+1$ variables are $(0, \dots , 0, 0, 2), \; (0, \dots , 0, 1, 1), \; (0, \dots , 0, 2, 0), \; 
(0, \dots , 1, 0, 1)$, and $(0, \dots , 1, 1, 0)$. The corresponding Tsfasman-Boguslavsky bound in \eqref{TBr} is given, respectively, by
$$
0, \quad 1, \quad 2, \quad q+1, \quad \text{and} \quad q+2. 
$$
One can check that these coincide with the corresponding Zanella bounds in \eqref{Zr}.  Moreover, it is easy to see that the bounds are attained by taking in the first case the set, say $\mathscr{Q}$ of all monomials of degree $2$ in $\Fq[x_0,x_1, \dots , x_m]$ and in the remaining four cases, taking the set obtained from  $\mathscr{Q}$ by successively dropping $x_m^2$, $x_{m-1}^2$, $x_{m-1}x_m$, and $x_{m-2}^2$. Using these observations together with Corollary \ref{cor1} we see as a special case that  the 
TBC 
is true for quadrics in $\PP^2$, i.e., \eqref{TBr} holds if $d=2$ and $m=2$. 
\end{remark}    

\begin{corollary}
\label{cor2}
If $d=2$ and $m > 2$, then \eqref{TBr} does not hold in general. In fact, \eqref{TBr} is false for at least $\binom{m-1}{2}$ values of 
positive integers $r$ with $m+1 < r \le \delta_m$.
\end{corollary}

\begin{proof}
Let $e_i$ ($1 \le i \le m+1$) be as in the proof of Corollary \ref{cor1}. Also let $k$ be the 
unique integer such that $-1\le k < m$ and $\delta_m - \delta_{k+1} < r \le \delta_{m} - \delta_{k}$. 
Write $r = \delta_m - \delta_{k+1} + i$ so that $1\le i \le k + 2$. Observe that the $r$th element, in descending lexicographic order, among the exponent vectors of monomials of degree $2$ in $m+1$ variables, is precisely $e_{m-k} + e_{m-k + i-1}$. In particular its first nonzero coordinate is in the position $j:=m-k$, and thus, with notation as in \eqref{pm}, the corresponding Tsfasman-Boguslavsky bound in \eqref{TBr} is given by 
\begin{equation*}
\begin{split}
T_r &= p_{m - 2j} + (p_{m- j} - p_{m - 2j}) + (p_{m - (j + i - 1)} - p_{m - j - (j+ i - 1)}) \\
&= p_{m-j} + p_{m - i - j + 1} - p_{m - i - 2j + 1} \\
& = p_k + p_{k - i + 1} - p_{2k - m - i + 1}. 
\end{split}
\end{equation*}
On the other hand, the corresponding Zanella bound in \eqref{Zr} is given by 
$$
Z_r = p_k + \lfloor q^{\epsilon -1} \rfloor, \ \text{ where }\;  \epsilon = \delta_m - \delta_k - (\delta_m - \delta_{k+1} + i) = \delta_{k+1} - \delta_k - i = k + 2 - i.
$$
It follows that  if $0 \le k < m -1$ and $1\le i \le k$, then 
$$
T_r - Z_r = p_{k-i} - p_{2k - m - i + 1} \ge q^{k-i} > 0,  
$$
and so
$Z_r < T_r$. Thus Theorem \ref{ZanellaThm} implies that $T_r$ can not be the maximum of 
$|V(F_1, \dots , F_r)|$ for arbitrary sets of $r$ linear independent   homogeneous polynomials of degree~$2$ in $\Fq[x_0, \dots , x_m]$. The number of such values of $r = \delta_m - \delta_{k+1} + i$ is
$$
\sum_{k=0}^{m-2} \sum_{i=1}^k 1 = \sum_{k=0}^{m-2} k = \frac{(m-1)(m-2)}{2} = \binom{m-1}{2}.
$$
Evidently this is positive if $m>2$. This proves the corollary.
\end{proof}
%

\begin{example}
The simplest case where the TBC 
is false seems to be that of intersections of 5 linearly independent quadrics in $\PP^3$. One can see in this case that Tsfasman-Boguslavsky bound is $T_5 = 2(1+q)$, whereas the Zanella bound is $Z_5 = 1+ 2q$, which is strictly smaller.
\end{example}

\begin{remark}
\label{ThreeConj} 
The Tsfasman-Boguslavsky Conjecture  stated in the Introduction (and abbreviated as TBC) is, in fact, a culmination of several conjectures that can be found in the paper of Boguslavsky  \cite{Bog}, with at least one of the conjectures ascribed to Tsfasman. More precisely, the TBC is Corollary~5 of \cite{Bog} whose hypothesis  is that Conjecture~3 of \cite{Bog} holds and whose ``proof" uses Lemma~4 of \cite{Bog}.  
For ease of reference, we state below 
Conjectures 1, 2 and 3 of  \cite{Bog}. 
To this end, let us first 
introduce some 
terminology. A projective variety $X$ in $\PP^m$ over $\Fq$ is said to be
\textit{linear} if  
its $\Fq$-rational points lie on the linear components of $X$,
An $m$-tuple $(\a_1, \a_2, \dots, \a_m) \in \mathbb{Z}^{m}$ is said to be the \textit{dim-type} of a projective variety $X$ in $\PP^m$ if $X$ has 
$\a_i$ irreducible components of codimension $i$ for $i = 1, 2, \dots, m$. 
Given a finite family $\mathscr{X}$ of projective varieties in $\PP^m$, 
an element $X$ of $\mathscr{X}$ 
is said to be:  
\begin{enumerate}
\item[(i)] \textit{maximal} in $\mathscr{X}$ if $|X (\Fq)| =  \max \{ |Y(\Fq)| : Y \in \mathscr{X} \}$, and 
\item[(ii)] \textit{dim-maximal} in $\mathscr{X}$ if the dim-type of $X$ is maximal (among the dim-types of all the elements of $\mathscr{X}$) with respect to the lexicographic order on $\Z^m$. 
\end{enumerate}
The 
conjectures in \cite{Bog} concern the family, say $\mathscr{X}_r$,  of projective varieties in $\PP^m$ defined by $r$ linearly independent homogeneous polynomials in $\Fq[x_0, x_1, \dots, x_m]$ of degree $d$, and are as follows. Here $\Lambda(d, m)$ is as in the Introduction. 
\begin{enumerate}
\item[1.] 
There exists a maximal family in $\mathscr{X}_r$ which is linear.
\item[2.] 
If $(\nu_1, \dots, \nu_{m+1})$ is the $r$-th element of $\Lambda(d, m)$ in descending lexicographic order,  then the dim-type of a dim-maximal element of $\mathscr{X}_r$ is $(\nu_1, \dots, \nu_m)$.
\item[3.] 
There exists a maximal family in $\mathscr{X}_r$ which is dim-maximal.
\end{enumerate}
We remark that our positive result Corollary \ref{cor1} and its proof (especially the examples therein) imply immediately that Conjectures 1 and 3 above 
hold in the affirmative
when $d=2$ and $r\le m+1$. Moreover, it is not difficult to also deduce Conjecture~2 
in this case. On the other hand, the negative result in Corollary \ref{cor2} does not necessarily imply that Conjectures 1, 2, and 3  above 
are false. 
It should also be remarked that the Tsfasman-Boguslavsky Conjecture stated in the Introduction has two aspects: (i) the expression on the right hand side of the equality in \eqref{TBr} is an upper bound for the number of common zeros of a system of $r$ linearly independent homogeneous polynomials of degree $d$ in $\Fq[x_0, x_1, \dots , x_m]$, and (ii) this upper bound is attained. Our negative result in Corollary \ref{cor2} shows only that 
(ii) is false, but does not rule out the possibility that  
(i) holds, in general. 
\end{remark}

\section{Applications and Supplements} 
\label{sec:codes}

In this first subsection below, we outline the relevance of TBC 
to coding theory, and in the second subsection, we provide a comparison with an older conjecture of Lachaud
\cite[Conj. 12.2]{GL} that is also stated, albeit with much too general hypothesis, by Boguslavsky \cite[Conj. 4]{Bog}, and 
settled recently by Couvreur \cite{C}. 

\subsection{Projective Reed Muller Codes}
Fix positive integers $m$, $d$ and let $n:= p_m$. Each point of $\PP^m(\Fq)$ admits a unique representative in $\Fq^{m+1}$ in which the first nonzero coordinate is $1$. Let 
$P_1, \dots, P_{n}$ be an ordered listing of such representatives in $\Fq^{m+1}$ of points of $\PP^m(\Fq)$.
Denote by $\Fq[x_0, \dots , x_m]_d$ the space of homogeneous polynomials in $\Fq[x_0, \dots , x_m]$ 
of degree $d$  together with the zero polynomial. Define 
$$
\PRM_q(d,m):=\{ \left( f(P_1), \dots , f(P_n)\right) : f\in \Fq[x_0, \dots , x_m]_d\}.
$$
Evidently, this is a linear subspace of $\Fq^n$, and hence a $q$-ary (linear) code of length $n$. It is called the projective Reed-Muller code. This code is analogous to a more widely studied class of codes called (affine or generalized) Reed-Muller code $\RM_q(d,m)$. See, for example, \cite{DGM, HP} for more on 
Reed-Muller codes and   \cite[Prop. 4]{BGH2} for a summary of several of its basic properties. The study of projective Reed-Muller codes was pioneered by Lachaud \cite{L1, L2} and S{\o}rensen \cite{So}.  The relation with the TBC 
is through the notion of generalized Hamming weights, also known as higher weights, that goes back at least to Wei \cite{W}. In general, for any $q$-ary linear code $C$ of length $n$ and dimension $k$, 
and any $D\subseteq C$, one defines
$$ 
\w (D):= \left| \left\{i\in\{1, \dots , n\} : c_i\ne 0 \text{ for some } c\in D\right\} \right|.
$$
Now for $r=1, \dots ,k$, the 
$r$th \emph{higher weight} of  $C$ is
defined by  
$$
d_r (C) = \min \{ \w( D) : D \mbox{ is a subspace of $C$ with } 
\dim D = r\},
$$
It may be remarked that $d_1(C)$ is what is called the minimum distance of $C$. The relationship of $\PRM_q(d,m)$ with \eqref{TBr} is as follows: if $d< q$, then for 
$1\le r \le \binom{m+d}{d}$,  
\begin{equation}
\label{drTBC}
d_r\left(\PRM_q(d,m) \right) = p_m - \max_{F_1, \dots , F_r} |V(F_1, \dots , F_r) | , 
\end{equation}
where 
the maximum is 
over linearly independent $F_1, \dots , F_r$  in $\Fq[x_0, x_1, \dots , x_m]_d$.  
To see \eqref{drTBC} it suffices to use the relationship between linear codes and projective systems as described in \cite[Thm. 1.1.14]{TVN} and to note that when $d< q$, the code $\PRM_q(d,m)$ corresponds to the projective system given by the $\Fq$-rational points of the Veronese variety corresponding to the Veronese embedding of $\PP^m$ of degree $d$.  Thus it is clear that the 
TBC 
admits an equivalent statement in terms of an explicit formula for the higher weights of projective Reed-Muller codes. In particular, \eqref{drTBC} and Corollary \ref{cor1} 
imply that 
\begin{equation}
\label{drtwo}
d_r\left(\PRM_q(2,m) \right) =  q^m - \lfloor q^{m-r} \rfloor \quad  \text{ for } r=1, \dots , m+1 
\end{equation}
whereas Remark \ref{lastwt} shows that 
$$
d_{\delta_m - r}\left(\PRM_q(2,m) \right) = \begin{cases} p_m -  r & \text{ if } r=0, 1, 2,  \\
p_m - (q+r -2) & \text{ if } r = 3, 4. \end{cases}
$$
It may be noted that \eqref{drtwo} can be viewed as a generalization of the last theorem in \cite{L1}. We end this subsection by remarking that an affine analogue of the Tsfasman-Boguslavsky Conjecture is true, in general for $1< d < q$, thanks to 
the complete determination of all higher weights of the Reed-Muller code $\RM_q(d,m)$ by 
Heijnen and Pelikaan \cite[Thm. 5.10]{HP}.

\subsection{Comparison with a Theorem of Couvreur} 
In a recent preprint \cite{C}, Couvreur has proved the following result,  answering as a special case a conjecture that goes back to Lachaud and stated in \cite[Conj. 12.2]{GL} (see also \cite[Conj. 5.3]{LR}). 

\begin{theorem}[Couvreur]
\label{Couv}
Let $X$ be a nondegenerate projective variety 
in $\PP^m$ defined over $\Fq$. Suppose the irreducible components of $X$ have 
dimensions $n_1, \dots , n_t$ and degrees ${\delta}_1, \dots , \delta_t$, respectively. If 
$n_i < m $ for all $i=1, \dots , t$, 
 then 
\begin{equation}
\label{c1}
 |X(\Fq)| \le p_{2n - m} + \sum_{i=1}^t \delta_i \left( p_{n_i} - p_{2n_i -m} \right) \quad \text{where} \quad n:= \max\{n_1, \dots , n_t\}.
\end{equation}
In particular, if $X$ 
is equidimensional 
of dimension $n$ and  degree $\delta$,   
then 
\begin{equation}
\label{c2}
 |X(\Fq)| \le {\delta} p_n - ({\delta}-1) p_{2n - m} =  {\delta}(p_n - p_{2n - m}) + p_{2n - m}.
\end{equation}
Moreover, the upper bound is optimal for equidimensional varieties. 
\end{theorem}


The original conjecture by Lachaud assumed $X$ be a complete intersection (and hence equidimensional) 
of degree $\delta \le q+1$ and 
had an additional hypothesis that $2n \ge m$, lest the bound in \eqref{c2} reduces to a known inequality (cf. \cite[Thm. 3]{Bog}, \cite[Prop. 12.1]{GL}). 
Just like the TBC, 
the conjecture by Lachaud reduces to Serre's 
inequality \eqref{SerreIneq} when codim $X = m - n =1$. But for codim$X > 1$, the relation between two conjectures above may not appear sufficiently clear and it may be worthwhile to try to make it clearer. First, it should be noted that the hypothesis of TBC is amenable to an easy verification---one just have to check that the defining equations have the same degree and are linearly independent. On the other hand, determining the dimensions and degrees of irreducible components from a given set of equations defining the variety can be quite difficult. In fact, even when the variety is known to be irreducible, determining the degree may not be easy, unless of course it is a hypersurface. One basic case where 
the hypotheses of the TBC and Theorem \ref{Couv} coincide and are easily checked 
is when $X \subseteq \PP^m$ is defined by the vanishing of $r$ linearly independent homogeneous polynomials in $m+1$ variables, each of the same degree $d$, and $n = \dim X = m-r$ so that $X$ is a complete intersection. In this case $X$ is equidimensional and has degree 
$\delta = d^r$.  Assume, for simplicity, that $n\ge 0$, i.e., $r \le m$ and that $d> 1$ and $\delta \le q+1$. 
Then $(d-1, 0, \dots, 0, 1, 0, \dots, 0)$, with $1$ is in the $r$-th place, 
 is the $r$th element of $\Lambda(d,m)$. 
Consequently, 
the Tsfasman-Boguslavsky bound, say $T_r(d)$, 
on $|X(\Fq)|$ is equal to 
$$
p_{m-2}+ (d-1)(p_{m-1} - p_{m-2}) + (p_{m-r} - p_{m-r-1}) = (d-1)q^{m-1} + q^{m-r} + p_{m-2}.
$$
Also since $d - 1 \ge 1$,  putting $ p_{m-2} = (q^{m-1} -1)/(q-1)$, we see that 
\begin{equation}
\label{Trdge}
T_r(d) \ge \frac{q^m + q^{m-r+1} - q^{m-r} -1}{q-1}.
\end{equation}
On the other hand, the Couvreur bound in \eqref{c2}, say $C_r(d)$, 
 in this situation is 
$$
d^r (p_{m-r} - p_{m-2r}) + p_{m -2r} =  (d^r -1) (p_{m-r} - p_{m-2r}) + p_{m-r}.
$$ 
Since 
$d^r   \le q+ 1$, i.e., $d^r-1 \le q$, we easily see that 
$$
C_r(d)  \le   \frac{q^{m - 2r+2} \left( q^{r} - 1 \right) +\left( q^{m - r+1} - 1\right) }{q-1} 
=  \frac{ q^{m - r+2}  +  q^{m - r+1} -  q^{m -2 r+2} - 1 }{q-1}  .
$$
Comparing the right hand side of the above equation with \eqref{Trdge}, we see that $C_r(d) \le T_r(d)$ if $r\ge 2$. 
It follows that the Couvreur bound is sharper, especially when $d>2$ and $r>2$. Of course, this, by itself, does not contradict the TBC since the projective varieties 
where the Tsfasman-Boguslavsky bounds are attained are seldom equidimensional (that is to say, having  all its irreducible components of the same dimension), let alone complete intersections. 
Indeed, in the commonly applicable situation considered above, projective varieties 
attaining the Tsfasman-Boguslavsky  bound is expected to have $d-1$ common components of codimension $1$ and one of codimension $r$. As Couvreur \cite[\S 5.2]{C} has remarked, his bound in the non-equidimensional case might not be optimal. This is, in fact, true, and to see 
this, one can consider $d=2$, $r\le m$, and the example of quadrics $Q_1, \dots , Q_r$ in the proof of Corollary \ref{cor1}. As we have seen, the projective variety, say $X$, cut out by these quadrics has $p_{m-1} + q^{m-r}$ 
points. Also it is clear that $X$ has two irreducible  components, the hyperplane $x_0=0$ and the linear subspace $x_1 = \dots = x_r =0$. Both the components are linear and are complete intersections. Thus in the notation of Theorem \ref{Couv}, we have $t=2$ and $n_1=m-1$, $n_2 = m-r$, while $\delta_1 = 1 = \delta_2$. It follows that the upper bound of Couvreur in \eqref{c1} in this case is 
$$
p_{2(m-1) -m} +  \left( p_{m-1} - p_{2(m-1) -m} \right) 
+  \left( p_{m-r} - p_{2(m-r) -m} \right), 
$$
or in other words,
$$
 p_{m-1} + q^{m-r} + q^{m-r-1} + \dots + q^{m-2r+1}. 
$$
So if $r\ge 2$, then the Tsfasman-Boguslavsky bound is sharper than the Couvreur bound in this case. 
It is thus seen that the two bounds compliment each other and neither implies the other, in general.

\section*{Acknowledgements}
We are grateful to: (i)  Masaaki Homma for bringing \cite{H} to our attention and also outlining the double counting argument given in Remark \ref{rem:ZanellaLemma} to prove Lemma \ref{Zanella}, (ii) Michael Tsfasman for some helpful discussions and suggesting that Lemma \ref{Zanella} could follow from the Plotkin bound, (iii) Olav Geil for pointing out that the bound for $|Z(f)|$ in \eqref{BasicBound} is also known as Schwartz-Zippel bound, and (iv) the anonymous referee for many useful comments on an earlier version of this article.

\end{document}